\documentclass[10pt,reqno]{amsart}

\usepackage{mathtools,amsthm,amssymb,mathrsfs,enumitem,wasysym,esint}
\usepackage{newtxtext}
\usepackage{scalerel}

\usepackage{newtxtext}

\usepackage[numbers,sort&compress]{natbib}
\usepackage{mathscinet}

\makeatletter
\def\@tocline#1#2#3#4#5#6#7{\relax
  \ifnum #1>\c@tocdepth 
  \else
    \par \addpenalty\@secpenalty\addvspace{#2}%
    \begingroup \hyphenpenalty\@M
    \@ifempty{#4}{%
      \@tempdima\csname r@tocindent\number#1\endcsname\relax
    }{%
      \@tempdima#4\relax
    }%
    \parindent\z@ \leftskip#3\relax \advance\leftskip\@tempdima\relax
    \rightskip\@pnumwidth plus4em \parfillskip-\@pnumwidth
    #5\leavevmode\hskip-\@tempdima
      \ifcase #1
       \or\or \hskip 1em \or \hskip 2em \else \hskip 3em \fi%
      #6\nobreak\relax
    \hfill\hbox to\@pnumwidth{\@tocpagenum{#7}}\par
    \nobreak
    \endgroup
  \fi}
\makeatother

\numberwithin{equation}{section}
\numberwithin{figure}{section}
\theoremstyle{plain}
\newtheorem{thm}{Theorem}[section]
  \theoremstyle{definition}
  
  \newtheorem*{defn*}{Definition}
  
  \theoremstyle{remark}
  \newtheorem*{rem*}{Remark}
  
  \theoremstyle{plain}
  \newtheorem{lem}[thm]{Lemma}
  \theoremstyle{plain}
  \newtheorem{prop}[thm]{Proposition}
  
  \theoremstyle{definition}

\global\long\def\dist{\mathrm{dist}}

\makeatletter
\newcommand{\norm}{\@ifstar{\@normb}{\@normi}}
\newcommand{\@normb}[2]{\left\Vert{#1}\right\Vert_{#2}}
\newcommand{\@normi}[2]{\Vert{#1}\Vert_{#2}}
\makeatother

\global\long\def\Sob#1#2{W^{#1,#2}}
\global\long\def\Leb#1{L^{#1}}
\global\long\def\Lor#1#2{L^{#1,#2}}

\DeclareMathOperator{\supp}{supp}

\global\long\def\oSob#1#2{{W}^{#1,#2}_{0}}

\global\long\def\loc{\mathrm{loc}}

\newcommand{\action}[1]{\left<#1 \right>}
\newcommand{\boldb}{\mathbf{b}}

\newcommand{\boldF}{\mathbf{F}}
\newcommand{\boldG}{\mathbf{G}}

\newcommand{\myd}[1]{\,d#1}

\DeclareMathOperator{\Div}{div}

\newcommand{\relphantom}[1]{\mathrel{\phantom{#1}}}

\title[$W^{1,2+\varepsilon}$-result for elliptic equation with singular drifts]{Existence and uniqueness of weak solution in $W^{1,2+\varepsilon}$ for elliptic equation with drifts in weak-$L^{n}$ spaces}
\author{Hyunwoo Kwon}
\address{(28187) Department of Mathematics, Republic of Korea Air Force Academy, Postbox 335-2, 635, Danjae-ro   Sangdang-gu, Cheongju-si  Chung\-cheong\-buk-do, Republic of Korea}
\email{willkwon@sogang.ac.kr; willkwon@afa.ac.kr}

\date{\today}

\allowdisplaybreaks

\begin{document}
\begin{abstract}
We consider the following Dirichlet problems for elliptic equations with singular drift $\mathbf{b}$:
\[ \text{(a) } -\operatorname{div}(A \nabla  u)+\operatorname{div}(u\mathbf{b})=f,\quad \text{(b) } -\operatorname{div}(A^T \nabla v)-\mathbf{b} \cdot \nabla v =g \quad \text{in } \Omega, \]
where $\Omega$ is a bounded Lipschitz domain in $\mathbb{R}^n$, $n\geq 2$. Assuming that $\mathbf{b}\in L^{n,\infty}(\Omega)^n$ has non-negative weak divergence in $\Omega$, we establish existence and uniqueness of weak solution in  $W^{1,2+\varepsilon}_0(\Omega)$ of the problem (b) when $A$ is bounded and uniformly elliptic. As an application, we prove existence and uniqueness of weak solution $u$ in $\bigcap_{q<2} W^{1,q}_0(\Omega)$ for the problem (a) for every $f\in \bigcap_{q<2} W^{-1,q}(\Omega)$.
\end{abstract}
\maketitle

\section{Introduction}

Given a vector field $\boldb:\Omega\rightarrow \mathbb{R}^n$, we consider the following Dirichlet problems for elliptic equations of second-order:
\begin{equation}\label{eq:divergence-type}
\left\{\begin{alignedat}{2}
-\Div (A \nabla u)+\Div(u\boldb)&=f &&\quad \text{in } \Omega,\\
u&=0 &&\quad \text{on } \partial\Omega,
\end{alignedat}
\right.
\end{equation}
and
\begin{equation}\label{eq:non-divergence-type}
\left\{\begin{alignedat}{2}
-\Div(A^T \nabla v)-\boldb \cdot \nabla v&=g &&\quad \text{in } \Omega,\\
v&=0&&\quad \text{on } \partial\Omega,
\end{alignedat} 
\right.
\end{equation}
on a bounded  domain  $\Omega$  in $\mathbb{R}^n$, $n\geq 2$. Here $A=(a^{ij})_{1\leq i,j\leq n} : \mathbb{R}^n \rightarrow \mathbb{R}^{n^2}$ is a given matrix valued function. We assume that $A$ is bounded and uniformly elliptic, that is, there exist $0<\delta<1$ and $K>1$ such that 
\begin{equation}\label{eq:uniformly-elliptic}
\delta |\xi|^2 \leq A(x)\xi \cdot \xi \quad \text{and}\quad |A(x)|\leq K\quad \text{for all } \xi \in \mathbb{R}^n\quad \text{and}\quad \text{a.e. } x\in \mathbb{R}^n.
\end{equation}

Suppose that $g\in \Sob{-1}{q}(\Omega)$ for some $1<q<\infty$. For $1<p<\infty$, we say that $v\in \oSob{1}{p}(\Omega)$ is a weak solution of the problem \eqref{eq:non-divergence-type} if 
\begin{equation}\label{eq:weak-sol-nondiv}
    \boldb\cdot \nabla v \in \Leb{1}(\Omega)^n\quad \text{and}\quad \int_\Omega (A^T\nabla v)\cdot \nabla \phi  - (\boldb \cdot \nabla v)\phi \myd{x}=\action{g,\phi} 
\end{equation} 
for all $\phi \in C_c^\infty(\Omega)$. Here $C_c^\infty(\Omega)$ denotes the set of all smooth functions with compact support in $\Omega$. Weak solution for \eqref{eq:divergence-type} can be similarly defined.

When $\boldb=\mathbf{0}$, Meyers \cite{M63} proved that if $\Omega$ is a bounded smooth  domain in $\mathbb{R}^n$, then there exists $1<p_0<2$ such that for $p_0<p<p_0'$, if $f\in \Sob{-1}{p}(\Omega)$, then there exists a unique weak solution $u\in \oSob{1}{p}(\Omega)$ of \eqref{eq:divergence-type}. Moreover, we have 
\[   \norm{\nabla u}{\Leb{p}(\Omega)} \leq C \norm{f}{\Sob{-1}{p}(\Omega)} \]
for some constant $C=C(n,\delta,p,K,\Omega)$. Here $p_0'$ is the H\"older conjugate to $p_0$ defined by $p_0'=p_0/(p_0-1)$. For several decades, many authors have studied  $\Sob{1}{p}$-estimate for problems \eqref{eq:divergence-type} and \eqref{eq:non-divergence-type} under various assumptions on $A$ and $\Omega$  with bounded vector fields $\boldb:\Omega \rightarrow \mathbb{R}^n$, see \cite{AQ, B, BW, DK} e.g.

One naturally asks the solvability of the problems \eqref{eq:divergence-type} and \eqref{eq:non-divergence-type} when $\boldb$ is unbounded. To explain the results, we fix a number $r$ so that $n\leq r<\infty$ if $n\geq 3$ and $2<r<\infty$ if $n=2$. In Ladyzhenskaya-Ural'tseva \cite[Theorems 5.3, 5.5, Chapter 3]{LU68}, they considered the following Dirichlet problem 
\[   -\Div(A\nabla u +u\boldb)+\mathbf{c}\cdot \nabla u=f+\Div \boldF\quad \text{in } \Omega,\quad u=0\quad \text{on } \partial\Omega \]
where  $A$ satisfies \eqref{eq:uniformly-elliptic}, $\boldb,\mathbf{c} \in \Leb{r}(\Omega)^n$, and $\Omega$ is a bounded Lipschitz domain in $\mathbb{R}^n$.   It was shown that for every $f \in \Leb{2\hat{n}/(\hat{n}+2)}(\Omega)$ and $\boldF \in \Leb{2}(\Omega)^n$, there exists a unique weak solution $u\in \oSob{1}{2}(\Omega)$ of the problem. Here $\hat{n}=n$ if $n\geq 3$ and $\hat{n}=2+\varepsilon$ for some $\varepsilon>0$.  Stampacchia \cite{S65} also considered similar problem   under the assumption that $f\in \Leb{2}(\Omega)$, $\boldF \in \Leb{2}(\Omega)^n$, $\boldb, \mathbf{c}\in \Leb{n}(\Omega)^n$, $n\geq 3$, see \cite[Th\'eor\`eme 3.4]{S65} for details.  Later, Droniou \cite{D02} gave another proof for $\Sob{1}{2}$-estimates of weak solutions of problems \eqref{eq:divergence-type} and \eqref{eq:non-divergence-type} when $A$ satisfies \eqref{eq:uniformly-elliptic} and  $\boldb \in \Leb{r}(\Omega)^n$ on a bounded Lipschitz domain in $\mathbb{R}^n$. 

An extension to $\Sob{1}{p}$-result with $p\neq 2$ was first shown by  Kim-Kim \cite{KK15}. In that paper,  $\Sob{1}{p}$-estimate was established for \eqref{eq:divergence-type} when $1<p<r$, $A=I$,  $\boldb \in \Leb{r}(\Omega)^n$, and $\Omega$ is bounded $C^1$-domain in $\mathbb{R}^n$. Later, Kang-Kim \cite{KK17} extend the result of Kim-Kim \cite{KK15} to the case when $A$ has partially small BMO in the sense of Dong-Kim \cite{DK} and $\Omega$ has a small Lipschitz constant. Another extensions including Bessel potential spaces were shown by Kim and the author \cite{KK18} when $A=I$, $\boldb \in \Leb{n}(\Omega)^n$, and $\Omega$ is arbitrary bounded Lipschitz domain in $\mathbb{R}^n$,  $n\geq 3$. Similar results also hold for the dual problem \eqref{eq:non-divergence-type}. Results on parabolic equations with unbounded drifts were recently obtained by Kim-Ryu-Woo \cite{KRW20}.

Note that the space $\Leb{n}(\Omega)^n$ is optimal among $\Leb{r}(\Omega)^n$, $n\geq 3$ for the drift $\boldb$ to guarantee unique solvability for problems \eqref{eq:divergence-type} and \eqref{eq:non-divergence-type}.   Indeed, let $A=I$, $\Omega$ be a unit ball centered at the origin, $\boldb(x)=(2-n)x/|x|^2$, and $v(x)=\ln |x|$. Then $v\in \oSob{1}{2}(\Omega)$ is a non-trivial weak solution of \eqref{eq:non-divergence-type} {with $g=0$}. Note that $\boldb \in \Leb{q}(\Omega)^n$ for any $q<n$, in particular, $\boldb \in \Lor{n}{\infty}(\Omega)^n$. On the other hand, as pointed out by Monscariello \cite{M11}, $p$-weak solution of \eqref{eq:divergence-type} may fail to exist for general $\boldb \in \Lor{n}{\infty}(\Omega)^n$. Therefore, in order to discuss unique solvability of weak solutions for \eqref{eq:divergence-type} and \eqref{eq:non-divergence-type}, we need to assume additional condition on $\boldb\in \Lor{n}{\infty}(\Omega)^n$.  

There are few results on unique solvability of weak solutions for \eqref{eq:divergence-type} and \eqref{eq:non-divergence-type} when $\boldb$ belongs to weak $\Leb{n}$-spaces. It was first shown by Monscariello \cite{M11} that if $\boldb\in \Lor{n}{\infty}(\Omega)^n$ has sufficiently small $\norm{\boldb}{\Lor{n}{\infty}}$ and $n\geq 3$, then problem \eqref{eq:divergence-type} has a unique weak solution in $\oSob{1}{p}(\Omega)$ for $p$ sufficiently close to $2$. Motivated partially by fluid mechanics, Kim-Tsai \cite{KT20} established $\Sob{1}{p}$-estimate and  $\Sob{1}{p'}$-estimate for \eqref{eq:divergence-type} and \eqref{eq:non-divergence-type} when $2\leq p<n$, respectively, under the assumption that $A=I$, $n\geq 3$, and  $\Div \boldb \geq 0$ in $\Omega$. Here $\Div \boldb \geq 0$ in $\Omega$ means that 
\[  -\int_\Omega \boldb \cdot \nabla \phi \myd{x}\geq 0\quad \text{for all non-negative } \phi \in C_c^\infty(\Omega).  \] 
They also obtained $\Sob{1}{p}$-estimate and  $\Sob{1}{p'}$-estimate for \eqref{eq:divergence-type} and \eqref{eq:non-divergence-type}, respectively when $n'< p<2$ in a bounded $C^1$-domain in $\mathbb{R}^n$ under an additional assumption that $\Div \boldb \in \Lor{n/2}{\infty}(\Omega)$.   In the same paper, it was shown in Theorem 2.3 that if $\Omega$ is a bounded $C^{1,1}$-domain in $\mathbb{R}^n$, $n\geq 3$, and $\boldb\in \Lor{n}{\infty}(\Omega)^n$ satisfies $\Div \boldb \geq 0$ in $\Omega$ and $\Div \boldb \in \Lor{n/2}{\infty}(\Omega)$, then for given data $g\in \Sob{-1}{q}(\Omega)$ and $q>n$, any weak solution $v$ of \eqref{eq:non-divergence-type} is in $\oSob{1}{n+\varepsilon}(\Omega)$ for some $0<\varepsilon\leq q-n$, where $\varepsilon>0$  depends only on $n$, $p$, $\Omega$, $\norm{\boldb}{\Lor{n}{\infty}(\Omega)}$, and $\norm{\Div\boldb}{\Lor{n/2}{\infty}(\Omega)}$. For the dual problem, it was shown that if $f\in \Sob{-1}{p}(\Omega)$ for all $p<2$, then {there exists a unique $u\in \bigcap_{p<n'} \oSob{1}{p}(\Omega)$ satisfying \eqref{eq:divergence-type}.} 

Our paper is motivated by the recent paper of Chernobai-Shilkin \cite{CS19}.  They considered the equation \eqref{eq:non-divergence-type} in bounded $C^1$-domain  $\Omega \subset \mathbb{R}^2$ with $0\in \Omega$ when $A=I$  and $\boldb=\boldb_0+\alpha x/|x|^2$, where  $\alpha$ is a real number and $\boldb_0 \in \Lor{2}{\infty}(\Omega)^2$ satisfies $\Div \boldb_0 =0$ in $\Omega$.  {In the case of $\alpha\geq 0$,} they obtain existence and uniqueness of weak solution $v$ in $\oSob{1}{2+\varepsilon}(\Omega)$ for the problem \eqref{eq:non-divergence-type} for given data $g\in \Sob{-1}{q}(\Omega)$ with $q>2$,  which can be regarded as a partial extension of Theorem 2.3 in Kim-Tsai \cite{KT20}. 

The purpose of this paper is to extend results of  Chernobai-Shilkin \cite{CS19} and Kim-Tsai \cite{KT20}.   More precisely, under the assumption that $\Omega$ is a bounded Lipschitz domain in $\mathbb{R}^n$ ($n\geq 2$),  $A$ satisfies \eqref{eq:uniformly-elliptic}, and $\boldb \in \Lor{n}{\infty}(\Omega)^n$ which has non-negative divergence in $\Omega$, we will show that for every $g\in \Sob{-1}{q}(\Omega)$, $q>2$, there exists a unique weak solution $v\in \oSob{1}{2+\varepsilon}(\Omega)$ of \eqref{eq:non-divergence-type}, where $\varepsilon$ depends only on $\delta$,  $K$, $\norm{\boldb}{\Lor{n}{\infty}(\Omega)}$, and the Lipschitz character of $\Omega$. Without assuming $\Div \boldb$ has integrability condition, existence of weak solution of \eqref{eq:divergence-type} seems to be open for general $f\in \Sob{-1}{p}(\Omega)$, $p<2$. However, we will show that if $f\in \Sob{-1}{p}(\Omega)$ for all $p<2$, then the problem \eqref{eq:divergence-type} has a unique weak solution $u$ in $ \bigcap_{p<2} \oSob{1}{p}(\Omega)$. 

 Now we present main results of this paper. The first theorem concerns existence and uniqueness of weak solution in  $\oSob{1}{2+\varepsilon}(\Omega)$ for the problem \eqref{eq:non-divergence-type} for given $g\in \Sob{-1}{q}(\Omega)$, $q>2$.  

\begin{thm}\label{thm:main-theorem-1}
Let   $\Omega$ be a bounded Lipschitz domain in $\mathbb{R}^n$, $n\geq 2$. Assume that $A$ satisfies \eqref{eq:uniformly-elliptic}, and  $\boldb \in \Lor{n}{\infty}(\Omega)^n$, and $\Div \boldb \geq 0$ in $\Omega$.  Then for $q>2$, there exists $0<\varepsilon<q-2$ depending on {$n$, $q$,} $\delta$, $K$,  $\norm{\boldb}{\Lor{n}{\infty}(\Omega)}$, and the Lipschitz character of $\Omega$  such that if $g \in \Sob{-1}{q}(\Omega)$, then there exists a unique weak solution $v\in \oSob{1}{2+\varepsilon}(\Omega)$ of \eqref{eq:non-divergence-type}. Moreover, we have 
\[    \norm{\nabla v}{\Leb{2+\varepsilon}(\Omega)} \leq C \norm{g}{\Sob{-1}{q}(\Omega)}\]
for some constant $C$ depending  on $n,q,\delta,\Omega,K$, and $\norm{\boldb}{\Lor{n}{\infty}(\Omega)}$. 
\end{thm}
\begin{rem*}
(i) When $n=2$, Theorem \ref{thm:main-theorem-1} extends Theorem 2.3 (i) in Kim-Tsai \cite{KT20} and Theorem 1.1 in Chernobai-Shilkin \cite{CS19}. In contrast to Theorem 2.3 (i) in Kim-Tsai \cite{KT20}, Theorem \ref{thm:main-theorem-1} does not require  integrability condition on $\Div\boldb$. 

(ii) Suppose that {$A$ satisfies \eqref{eq:uniformly-elliptic},} $\Omega$ is a bounded smooth domain in $\mathbb{R}^n$ ($n\geq 3$), and $\boldb \in \Leb{n}(\Omega)^n$. Using Meyer's result \cite{M63} and an argument given in  Kang-Kim \cite{KK17}, one can show that there exists $1<p_0<2$ depending on $n$, $\Omega$, $\delta$, and $K$  such that for any $p_0<p<p_0'$ and $g\in \Sob{-1}{p}(\Omega)$, there exists a unique $v\in \oSob{1}{p}(\Omega)$ satisfying \eqref{eq:non-divergence-type}. Moreover, we have 
\[  \norm{v}{\Sob{1}{p}(\Omega)} \leq C\norm{g}{\Sob{-1}{p}(\Omega)} \]
for some constant $C=C(n,p,\boldb,K,\delta,\Omega)>0$.  A key tool to prove this result is the following $\varepsilon$-inequality:
\[  \norm{\boldb\cdot \nabla v}{\Sob{-1}{p}(\Omega)}\leq \varepsilon \norm{v}{\Sob{1}{p}(\Omega)}+C(\varepsilon,n,p,\Omega,\boldb) \norm{v}{\Leb{p}(\Omega)},  \]
see \cite{KK17} for details. However, lack of a density property in $\Lor{n}{\infty}(\Omega)$ prevents us to derive such $\varepsilon$-inequality for general $\boldb \in \Lor{n}{\infty}(\Omega)^n$. 
\end{rem*}   

For simplicity of presentation, let us define 
\[   \oSob{1}{p-}(\Omega) = \bigcap_{q<p} \oSob{1}{q}(\Omega)\quad \text{and}\quad \Sob{-1}{p-}(\Omega) = \bigcap_{q<p} \Sob{-1}{q}(\Omega). \]

As an application of Theorem \ref{thm:main-theorem-1}, we have the following unique solvability result on the problem \eqref{eq:divergence-type} in $\oSob{1}{2-}(\Omega)$ for given data $f\in \Sob{-1}{2-}(\Omega)$.
\begin{thm}\label{thm:main-theorem-2}
Let  $\Omega$ be a bounded Lipschitz domain in $\mathbb{R}^n$, $n\geq 2$. Assume that  $A$ satisfies \eqref{eq:uniformly-elliptic} and, $\boldb \in \Lor{n}{\infty}(\Omega)^n$, and $\Div \boldb \geq 0$ in $\Omega$. 
\begin{enumerate}
\item[\rm (i)] There exists $1<r<2$ close to $2$, depending on $\delta$, $K$, $n$, $\norm{\boldb}{\Lor{n}{\infty}(\Omega)}$, and Lipschitz character of $\Omega$ such that if $u\in \oSob{1}{r}(\Omega)$ satisfies 
\begin{equation} \label{eq:weak-solution-uniqueness}
 \int_{\Omega} (A \nabla u)\cdot \nabla \psi \myd{x}-\int_\Omega (u\boldb)\cdot \nabla \psi \myd{x}=0 
 \end{equation}
for all $\psi \in C^1(\overline{\Omega})$ with $\psi=0$ on $\partial\Omega$, then $u$ is identically zero in $\Omega$.  
\item[\rm (ii)]  For every $f\in \Sob{-1}{2-}(\Omega)$, there exists a unique $u\in \oSob{1}{2-}(\Omega)$ satisfying \eqref{eq:divergence-type}.   
\end{enumerate}
\end{thm} 
\begin{rem*}
It was shown in  Theorem 2.4 in \cite{KT20} that if $\Omega$ is bounded $C^{1,1}$-domain in $\mathbb{R}^n$ ($n\geq 3$), $A=I$, and $\boldb \in \Lor{n}{\infty}(\Omega)^n$ satisfies $\Div \boldb \geq 0$ in $\Omega$ and  $\Div \boldb \in \Lor{n/2}{\infty}(\Omega)$ in addition, then the problem \eqref{eq:divergence-type} has a unique weak solution in $\oSob{1}{n'-}(\Omega)$. Since $\Omega$ is bounded, it follows that $\oSob{1}{2-}(\Omega)\subset \oSob{1}{n'-}(\Omega)$ for $n\geq 2$. Hence, without assuming additional condition on $\Div \boldb$, we obtain a uniqueness result in small space $\oSob{1}{2-}(\Omega)$ than $\oSob{1}{n'-}(\Omega)$ for the problem \eqref{eq:divergence-type} when $n\geq 3$. When  $n=2$, Theorem \ref{thm:main-theorem-2} can be regarded as an extension of 2D version of Theorem 2.4 in \cite{KT20}. 
\end{rem*}

Let us summarize proofs of main theorems. To prove Theorem \ref{thm:main-theorem-1}, when $n\geq 3$, existence and uniqueness of weak solution $v\in \oSob{1}{2}(\Omega)$ of \eqref{eq:non-divergence-type} is easily proved by using the Lax-Milgram theorem. Then we improve $\nabla v\in \Leb{2+\varepsilon}(\Omega)$ for some $\varepsilon>0$ by proving   reverse H\"older estimates for the gradient of weak solutions which have  self-improving property, observed first by Gehring \cite{G73} and extended by Giaquinta-Modica \cite{GM} (see \cite{G83}) Then by the uniqueness of weak solution in $\oSob{1}{2}(\Omega)$, this implies the desired result when $n\geq 3$.   When $n=2$, further scrutiny is necessary since $\boldb \cdot \nabla v\notin \Sob{-1}{2}(\Omega)$ for $\boldb \in \Lor{2}{\infty}(\Omega)^n$ and $v\in \oSob{1}{2}(\Omega)$ in general. Instead,  we first prove an existence of weak solution $v_k\in \oSob{1}{q}(\Omega_k)$ of \eqref{eq:non-divergence-type} with smooth  $A_k$ and $\boldb_k$ on smooth domain $\Omega_k$ satisfying $\overline{\Omega}_k\subset \Omega$ (see Proposition \ref{prop:approximation}  for construction).   Then we show that  there exists $2<p<q$ such that $\nabla v_k\in \Leb{p}(\Omega)$ satisfying $\norm{v_k}{\Sob{1}{p}(\Omega_k)}\leq C$, where $p$ and $C$ are independent of $k$.  Then the existence of weak solution follows by a standard compactness argument. Uniqueness follows by a standard energy estimate and the Poincar\'e inequality.    A proof of Theorem \ref{thm:main-theorem-2} follows by a duality argument. 

The rest of the paper is organized as follows. Section \ref{sec:prelim} is devoted to giving and proving some preliminary results including a basic theory on Lorentz spaces and $\Sob{1}{2}$-estimates for weak solutions of \eqref{eq:divergence-type} and \eqref{eq:non-divergence-type} when $n\geq 3$. In Section \ref{sec:reverse-Holder}, we will prove reverse H\"older estimates which play a crucial role in the proof of Theorem \ref{thm:main-theorem-1}. We present proofs of  Theorems  \ref{thm:main-theorem-1} and \ref{thm:main-theorem-2} in Section \ref{sec:proofs}.

 The introductory section is completed by introducing several notations used in this article. As usual, $\mathbb{R}^n$ stands for the Euclidean space of points $x=(x_1,\dots,x_n)$. For $i=1,\dots,n$, and a function $u(x)$, we set 
 \[ u_{x_i}=D_i u =\frac{\partial u}{\partial x_i},\quad \nabla u=(u_{x_1},\dots,u_{x_n}).  \]
  We denote  the Euclidean ball of radius $r$ whose center is $x$ by $B_r(x)= \{ y \in \mathbb{R}^n : |y-x|<r\}$. For an open set $U\subset \mathbb{R}^n$, by $C_c^\infty(U)$, we denote the set of infinitely differentiable functions with compact support in $U$. For a Borel set $A \subset \mathbb{R}^n$, we use $|A|$ to denote its Lebesgue measure. If $|A|<\infty$, we write 
 \[   (f)_A=\fint_{A} f\myd{x}=\frac{1}{|A|} \int_A f \myd{x}.\]
 For $1<p<\infty$, we write $p'$ the H\"older conjugate to $p$ defined by $p/(p-1)$. For $1<p<n$, $p^*=np/(n-p)$  denotes the Sobolev conjugate of $p$. We denote by $X'$ the dual space of a Banach space $X$. The dual pairing of $X$ and $X'$ is denoted by $\action{\cdot,\cdot}_{X',X}$ or simply $\action{\cdot,\cdot}$.  Finally, by  $C=C(p_1,\dots,p_k)$, we denote a generic positive constant depending only on the parameters $p_1,\dots,p_k$. 

\subsection*{Acknowledgment}
The author would like to thank his advisor Prof. Hyunseok Kim  and Prof. {Doyoon} Kim for valuable comments.  Also, the author would like to thank 2d Lt.~Hyunjip Kim for translating Th\'eor\`eme 3.4 in Stampacchia \cite{S65} into English.  Finally, the author would like to thank anonymous referees for their careful reading of the manuscript and giving useful comments and suggestions to improve the paper.

\section{Preliminaries}\label{sec:prelim}

In this section, we collect some preliminary results including standard results on Lipschitz domains, Lorentz spaces, basic estimates, and $\Sob{1}{2}$-results for problems \eqref{eq:divergence-type} and \eqref{eq:non-divergence-type} when $n\geq 3$. 
 
\subsection{Lipschitz domain}

This subsection is based on the classical paper Verchota \cite{V84} (see also \cite{TOB13,T17,V82,MT99}). A bounded domain $\Omega$ in $\mathbb{R}^n$ is said to be \emph{Lipschitz} if for each $x\in \partial\Omega$, there exist a rectangular coordinate system, $(x',x_n)$, $x'\in \mathbb{R}^{n-1}$, $x_n \in \mathbb{R}$, a neighborhood $U=U(x)\subset \mathbb{R}^n$ containing $x$ and a Lipschitz continuous function $\varphi_x = \varphi : \mathbb{R}^{n-1}\rightarrow \mathbb{R}$ satisfying $\Omega \cap U = \{ (x',x_n) \in \mathbb{R}^n : x_n>\varphi(x')\}\cap U$. 

By a \emph{cylinder} $Z=Z_r(x)$, we mean an open, right circular, doubly truncated cylinder centered at $x\in \mathbb{R}^n$ with radius equal to $r$. We say that $Z=Z_r(x)$, $x\in \partial\Omega$ is a \emph{coordinate cylinder} for $\Omega$ if
\begin{enumerate}
\item[(i)] the bases of $Z$ are some positive distance from $\partial\Omega$;
\item[(ii)] there exists a rectangular coordinate system for $\mathbb{R}^n$, $(x',x_n)$, $x'\in\mathbb{R}^{n-1}$, $x_n \in \mathbb{R}$, with $x_n$-axis containing the axis of $Z$;
\item[(iii)] there exists a Lipschitz continuous function $\varphi =\varphi_Z : \mathbb{R}^{n-1}\rightarrow \mathbb{R}$ such that $Z\cap \Omega = Z\cap \{(y',y_n) \in \mathbb{R}^n : y_n>\varphi(x')\}$ and $x=(0,\varphi(0))$. 
\end{enumerate}
The pair $(Z,\varphi)$ is said to be a \emph{coordinate pair}.   Under this coordinate system, we may take our coordinate cylinder $Z_r(x)=\{ y\in \mathbb{R}^n : |y'-x'|<r,\, |y_n-x_n|<(1+M)r\}$, where $M$ is the Lipschitz constant of $\varphi$. 

Since $\partial\Omega$ is compact, we fix a covering of the boundary $\partial\Omega$ by coordinate cylinders $\{Z_{r_i}(x_i)\}_{i=1}^N$ so that for $1\leq i\leq N$, $(Z_i,\varphi_i)$ and $(Z_i^*,\varphi_i$) are coordinate pair, where   $Z^*_{r_i}(x_i)=Z_{100r_i \sqrt{1+M^2}}(x_i)$  and $M = \max_{1\leq i\leq N} \norm{\nabla \varphi_i}{\Leb{\infty}(\mathbb{R}^{n-1})}$.  Define $R_0=\min \{r_i : i=1,\dots,N\}$.  The totality of parameters 
\[   Z_k,\varphi_k, \norm{\nabla \varphi_k}{\Leb{\infty}(\mathbb{R}^{n-1})}, r_k, 1\leq k\leq N \]
are said to determine the \emph{Lipschitz character} of $\Omega$. 

It is easy to see that for each $x\in \partial \Omega$, we have  $B_{R_0}(x)\subset Z^*_{r_i}(x_i)$ for some $i$. From this, we can show that there exists a constant $c>0$ depending on the Lipschitz character of $\Omega$ such that
\begin{equation}\label{eq:Lipschitz-domain-exterior-ball}
 |B_{R}(x)\setminus \Omega|\geq c|B_R(x)|  
\end{equation}
for all $x\in \partial \Omega$ and $0<R\leq R_0$.

The following approximation of  bounded Lipschitz domain by smooth domains can be found in Verchota \cite[Theorem 1.12]{V84}, see Appendix of Verchota \cite{V82} for the proof.
\begin{thm}\label{thm:Verchota}
Let $\Omega$ be a bounded Lipschitz domain in $\mathbb{R}^n$, $n \geq 2$.   There exist  a covering of $\partial \Omega$ by coordinate cylinders $Z$ and  sequence of $C^\infty$-domains $\{\Omega_j\}$ with $\overline{\Omega_j}\subset \Omega$ such that given a coordinate pair $(Z,\varphi)$ of $\partial\Omega$, then for each $j$, $Z^*\cap \partial\Omega_j$ is given as the graph of a $C^\infty$-function $\varphi_j$   such that $\varphi_j\rightarrow \varphi$ uniformly, $\norm{\nabla \varphi_j}{\Leb{\infty}(\mathbb{R}^{n-1})}\leq \norm{\nabla \varphi}{\Leb{\infty}(\mathbb{R}^{n-1})}$, and $\nabla \varphi_j\rightarrow \nabla \varphi$ pointwise a.e. 
\end{thm}

\subsection{Lorentz spaces and basic estimates for drifts terms} 

Let $\Omega$ be any domain in $\mathbb{R}^n$, $n\geq 2$.  For a measurable function $f: \Omega \rightarrow \mathbb{R}$, the \emph{distribution function} of $f$ is defined by 
\[    \mu_f(t) = |\{ x\in \Omega : |f(x)|>t \}|\quad (t\geq 0). \] 
The decreasing rearrangement of $f$, denoted by $f^*(t)$, is defined by 
\[  f^*(t) = \inf \{ \alpha \geq 0 : \mu_f(\alpha)\leq t\}. \]
For $1<p< \infty$ and $1\leq q\leq \infty$,  we define 
\[   \norm{f}{\Lor{p}{q}(\Omega)}  = \begin{cases}
 \left(\int_0^\infty \left(t^{1/p} f^*(t) \right)^q \frac{dt}{t} \right)^{1/q} & \quad \text{if } q<\infty, \\
 \sup_{t>0} t^{1/p} f^*(t) & \quad \text{if } q=\infty.
\end{cases} \]
The set of all $f:\Omega\rightarrow \mathbb{R}$ with $\norm{f}{\Lor{p}{q}(\Omega)}<\infty$ is denoted by $\Lor{p}{q}(\Omega)$ and is called the \emph{Lorentz space} with indices $p$ and $q$.  

The following H\"older inequality and refined Sobolev inequality in Lorentz spaces can be found in standard literatures (see e.g.\ \cite{AF,KKP15} and references therein). 
\begin{prop}\label{prop:Holder-Sobolev}\leavevmode
\begin{enumerate}
    \item[\rm (i)] Let $f\in \Lor{p_1}{q_1}(\Omega)$ and $g\in \Lor{p_2}{q_2}(\Omega)$, where $1< p_1,p_2<\infty$ satisfy $1/p=1/p_1+1/p_2\leq 1$. Assume further that $1/q_1+1/q_2\geq 1$ if $q=1$. Then 
    \[   fg\in \Lor{p}{q}(\Omega)\quad \text{and}\quad \norm{fg}{\Lor{p}{q}(\Omega)} \leq C(p)\norm{f}{\Lor{p_1}{q_1}(\Omega)} \norm{g}{\Lor{p_2}{q_2}(\Omega)} \]
    for any $q\geq 1$ with $1/q_1+1/q_2\geq 1/q$. 
    \item[\rm (ii)] Let $1\leq p<n$. Then for every $u\in \Sob{1}{p}(\mathbb{R}^n)$, we have 
    \[  u\in \Lor{p^*}{p}(\mathbb{R}^n)\quad \text{and}\quad \norm{u}{\Lor{p^*}{p}(\mathbb{R}^n)} \leq C(n,p)\norm{\nabla u}{\Leb{p}(\mathbb{R}^n)}.  \]
\end{enumerate}
\end{prop}

 Using H\"older's inequality and refined Sobolev's inequality in Lorentz spaces, we can prove the following basic estimate. 
\begin{prop}\label{prop:basic-estimates}
Let $\Omega$ be a bounded Lipschitz domain in $\mathbb{R}^n$, $n\geq 2$ and let $\boldb \in \Lor{n}{\infty}(\Omega)^n$.
\begin{enumerate}
  \item[\rm (i)] Suppose that the pair $(p,q)$ satisfies either 
    \begin{enumerate}
    \item[\rm (a)] $1<p<n$ and $q=p'$ or 
    \item[\rm (b)] $p=n$ and $q>n$. 
    \end{enumerate}
    Then there exists a constant $C=C(n,p,q,\Omega)>0$ such that
    \begin{equation}\label{eq:trilinear-estimate}
    \int_\Omega |(u\boldb)\cdot \nabla v|\myd{x}\leq C \norm{\boldb}{\Lor{n}{\infty}(\Omega)}\norm{u}{\Sob{1}{p}(\Omega)}\norm{v}{\Sob{1}{q}(\Omega)} 
        \end{equation} 
      for every  $u\in \Sob{1}{p}(\Omega)$ and $v\in \Sob{1}{q}(\Omega)$. 
  \item[\rm (ii)] Suppose in addition that $\Div \boldb \geq 0$ in $\Omega$. If $p>n$, then 
  \begin{equation}\label{eq:coercivity}
    -\int_\Omega u\boldb \cdot \nabla u \myd{x}\geq 0 
\end{equation}  
holds for all $u\in \oSob{1}{p}(\Omega)$. If $n\geq 3$, then \eqref{eq:coercivity} holds for all $u\in \oSob{1}{2}(\Omega)$. 
\end{enumerate} 
\end{prop}
\begin{proof}
(i) The case (a) can be found in \cite[Lemma 3.5]{KT20}. Suppose that $(p,q)$ satisfies (b). Choose $(s,r)$ satisfying $1<s<n'$ and  $nq/(q-n)<r<\infty$ so that  
\[  1=\frac{1}{q}+\frac{1}{s}+\frac{1}{r}.  \]
Since $\Omega$ is bounded and $n'\leq n$, the space $\Lor{n}{\infty}(\Omega)$ is continuously embedded into $\Leb{s}(\Omega)$. Also, by the Sobolev embedding theorem, $\Sob{1}{n}(\Omega)$ is continuously embedded into $\Leb{r}(\Omega)$. Hence by the H\"older inequality, we have 
\begin{align*} 
  \int_\Omega |(u\boldb)\cdot \nabla v|\myd{x}&\leq \norm{\boldb}{\Leb{s}(\Omega)}\norm{u}{\Leb{r}(\Omega)}\norm{\nabla v}{\Leb{q}(\Omega)}\\
  &\leq C \norm{\boldb}{\Lor{n}{\infty}(\Omega)}\norm{u}{\Sob{1}{n}(\Omega)}\norm{v}{\Sob{1}{q}(\Omega)}.  
\end{align*}
This completes the proof of (i).

(ii) Suppose that $p>n$. Since $\Div \boldb \geq 0$ in $\Omega$, it follows that 
\[ 
  -\int_\Omega (u\boldb)\cdot \nabla u \myd{x}=-\int_\Omega \boldb\cdot \nabla \left(\frac{1}{2} u^2 \right)\geq 0  
\]
for all $u\in C_c^\infty(\Omega)$. Since $C_c^\infty(\Omega)$ is dense in $\oSob{1}{p}(\Omega)$, it follows that for each $u\in \oSob{1}{p}(\Omega)$, there exists a sequence of functions $\{u_k\}$ in $C_c^\infty(\Omega)$ such that $u_k \rightarrow u$ in $\Sob{1}{p}(\Omega)$. Hence by (i), we have
\begin{align*}
  -\int_\Omega (u\boldb)\cdot \nabla u \myd{x}&=-\lim_{k\rightarrow \infty} \int_\Omega (u_k \boldb)\cdot \nabla u_k \myd{x}\geq 0.
\end{align*}
This completes the proof when $p>n$. {If $n\geq 3$ and $u\in \oSob{1}{2}(\Omega)$,  then \eqref{eq:coercivity} follows from}  \cite[Lemma 3.6]{KT20}. This completes the proof of Proposition \ref{prop:basic-estimates}. 
\end{proof}

The following result can be proved by using the Riesz representation theorem on $\Leb{p}(\Omega)$ and the classical Calder\'on-Zygmund estimate, see \cite[Lemma 3.9]{KT20} for the proof. 
\begin{prop}\label{prop:Riesz-representation}
Let $\Omega$ be a bounded domain in $\mathbb{R}^n$, $n\geq 2$, and let $1<p<\infty$. If $f\in \Sob{-1}{p}(\Omega)$, then there exists $\boldF \in \Leb{p}(\Omega)^n$ such that 
\[   \Div \boldF = f\quad \text{in } \Omega\quad \text{and}\quad \norm{\boldF}{\Leb{p}(\Omega)}\leq C \norm{f}{\Sob{-1}{p}(\Omega)}\]
for some constant $C=C(n,p,\Omega)>0$. 
\end{prop}

\subsection{Existence and uniqueness of weak solutions in $\oSob{1}{2}(\Omega)$ when $n\geq 3$}

When $n\geq 3$, existence and uniqueness of weak solutions in $\oSob{1}{2}(\Omega)$ for \eqref{eq:divergence-type} and \eqref{eq:non-divergence-type} are easily obtained by applying the Lax-Milgram theorem. 
\begin{prop}\label{prop:W12-estimates}
Let $\Omega$ be a bounded Lipschitz  domain in $\mathbb{R}^n$, $n\geq 3$. Suppose that  $\boldb \in \Lor{n}{\infty}(\Omega)^n$ satisfies $\Div \boldb \geq 0$ in $\Omega$. 
\begin{enumerate}
    \item[\rm (i)] For every $f\in \Sob{-1}{2}(\Omega)$, there exists a unique weak solution $u\in \oSob{1}{2}(\Omega)$ of \eqref{eq:divergence-type}. Moreover, we have 
    \[   \norm{u}{\Sob{1}{2}(\Omega)} \leq C \norm{f}{\Sob{-1}{2}(\Omega)} \]
    for some constant $C=C(n,\delta,\Omega)>0$. 
    \item[\rm (ii)] For every $g\in \Sob{-1}{2}(\Omega)$, there exists a unique weak solution $v\in \oSob{1}{2}(\Omega)$ of \eqref{eq:non-divergence-type}. Moreover, we have 
    \[   \norm{v}{\Sob{1}{2}(\Omega)} \leq C \norm{g}{\Sob{-1}{2}(\Omega)}\]
    for some constant $C=C(n,\delta,\Omega)>0$. 
\end{enumerate}
\end{prop}
\begin{proof}
Define $\mathcal{B}:\oSob{1}{2}(\Omega)\times \oSob{1}{2}(\Omega)\rightarrow \mathbb{R}$ by 
\[   \mathcal{B}(u,v)=\int_\Omega (A\nabla u -u\boldb) \cdot \nabla v \myd{x}. \]
Since $a^{ij}\in \Leb{\infty}(\mathbb{R}^n)$ and $\boldb \in \Lor{n}{\infty}(\Omega)^n$, the bilinear form $\mathcal{B}$ is bounded by Proposition \ref{prop:basic-estimates}. Moreover, it follows from Proposition \ref{prop:basic-estimates}, uniform ellipticity of $A$, and the Poincar\'e inequality that there exists a constant $C=C(n,\Omega)>0$ such that 
\[ \int_\Omega (A^T \nabla u-u\boldb) \cdot \nabla u \myd{x}\geq \delta \norm{\nabla u}{\Leb{2}(\Omega)}^2 \geq C \delta \norm{u}{\Sob{1}{2}(\Omega)}^2  \]
for all $u\in \oSob{1}{2}(\Omega)$. Hence by the Lax-Milgram theorem, for each $f\in \Sob{-1}{2}(\Omega)$, there exists a unique $u\in \oSob{1}{2}(\Omega)$ such that 
\[     \mathcal{B}(u,v)=\action{f,v}\quad \text{for all } v \in \oSob{1}{2}(\Omega).\]
Note that by the Poincar\'e inequality, we have
\[   \left|\action{f,u}\right|\leq \norm{f}{\Sob{-1}{2}(\Omega)}\norm{u}{\Sob{1}{2}(\Omega)} \leq C \norm{f}{\Sob{-1}{2}(\Omega)}\norm{\nabla u}{\Leb{2}(\Omega)} \]
for some constant $C=C(n,\Omega)>0$. Hence it follows from Young's inequality that 
\[    \delta\norm{\nabla u}{\Leb{2}(\Omega)} \leq \frac{\delta}{2}\norm{\nabla u}{\Leb{2}(\Omega)} + C\norm{f}{\Sob{-1}{2}(\Omega)}. \]
for some constant $C=C(n,\Omega,\delta)$. 
This proves (i). The proof of (ii) is similar so omitted. This completes the proof of Proposition \ref{prop:W12-estimates}. 
\end{proof}

\section{Reverse H\"older estimates} \label{sec:reverse-Holder}
In this section, we prove  reverse H\"older estimates for the gradient $\nabla v$ of a weak solution $v$ of the problem \eqref{eq:non-divergence-type}.  We use the following Poincar\'e-Sobolev's inequalities. 
\begin{prop}[Poincar\'e-Sobolev inequality]\label{prop:Poincare-Sobolev}
Let $R>0$ and let $1\leq p<n$. Then there exists a constant $C=C(n,p)>0$ such that 
\begin{equation}\label{eq:interior-Poincare}
    \norm{u-(u)_{B_R(x_0)}}{\Leb{p^*}(B_R(x_0))}\leq C \norm{\nabla u}{\Leb{p}(B_R(x_0))}\quad \text{for all } u\in \Sob{1}{p}(B_R(x_0)). 
    \end{equation}
Fix $\lambda>0$. Then there exists a constant $C=C(n,p,\lambda)>0$ such that 
\begin{equation}\label{eq:boundary-Poincare}
  \norm{u}{\Leb{p^*}(B_{R}(x_0))} \leq C \norm{\nabla u}{\Leb{p}(B_R(x_0))}  
\end{equation}
for all $u\in \Sob{1}{p}(B_R(x_0))$ satisfying 
\[   |\{ x\in B_R(x_0) : u(x)=0\}|\geq \lambda |B_R(x_0)|. \]
\end{prop}
\begin{proof}
Inequality \eqref{eq:interior-Poincare} can be found in standard literatures e.g. \cite{GT,E10}. Inequality \eqref{eq:boundary-Poincare} is  more or less standard result, but we give a proof for the sake of the completeness. By translation and scaling, it suffices to show the assertion when $R=1$ and $x_0$ is the origin. We first show that there exists a constant $C=C(n,p,\lambda)>0$ such that 
\begin{equation}\label{eq:boundary-Poincare-2}  
   \norm{u}{\Leb{p}(B_1)}\leq C \norm{\nabla u}{\Leb{p}(B_1)} 
\end{equation}
for all $u\in \Sob{1}{p}(B_1)$ with $|\{x\in B_1 : u(x)=0\}|\geq \lambda |B_1|$. Suppose on the contrary that inequality \eqref{eq:boundary-Poincare} is false. Then there exists a sequence $\{u_k\}$ such that $|\{ x\in B_1 : u_k(x)=0\}|\geq \lambda |B_1|$ but $\norm{u_k}{\Leb{p}(B_1)}=1$ and $\lim_{k\rightarrow \infty} \norm{\nabla u_k}{\Leb{p}(B_1)}=0$.  Then by Rellich-Kondrachov's theorem, we may assume that there exists $u_0 \in \Sob{1}{p}(B_1)$ such that $u_k\rightarrow u_0$ weakly in $\Sob{1}{p}(B_1)$ and strongly in $\Leb{p}(B_1)$. Since $\norm{\nabla u_0}{\Leb{p}(B_1)}=0$ and $\norm{u_0}{\Leb{p}(B_1)}=0$, it follows that $u_0$ is a nonzero constant a.e. on $B_1$. Hence 
\begin{align*}
0&=\lim_{k\rightarrow \infty} \int_{B_1} |u_k-u_0|^p \myd{x}\\
&\geq \lim_{k\rightarrow \infty} \int_{\{u_k=0\}} |u_k-u_0|^p \myd{x}\\
&\geq |u_0|^p \inf |\{u_k=0\}|>0.
\end{align*}
This contradiction leads us to prove \eqref{eq:boundary-Poincare-2}. Since $\Sob{1}{p}(B_1)$ is continuously embedded into $\Leb{p^*}(B_1)$, it follows that 
\[   \norm{u}{\Leb{p^*}(B_1)}\leq C(n,p)\left( \norm{u}{\Leb{p}(B_1)}+ \norm{\nabla u}{\Leb{p}(B_1)}\right) \leq C(n,p,\lambda) \norm{\nabla u}{\Leb{p}(B_1)}.\]
This completes the proof of Proposition \ref{prop:Poincare-Sobolev}. 
\end{proof}

From now on, we assume that $p_0=2$ if $n\geq 3$ and $p_0>2$ if $n=2$ in this section. We first derive an interior  reverse H\"older estimate for the gradient of weak solution. 
\begin{prop}\label{prop:interior-reverse-Holder}
Let $\Omega$ be a bounded Lipschitz domain in $\mathbb{R}^n$  {and $\boldG\in \Leb{2}(\Omega)^n$}. Assume that $A$ satisfies \eqref{eq:uniformly-elliptic}, $\boldb \in \Lor{n}{\infty}(\Omega)^n$, and $\Div \boldb \geq 0$ in $\Omega$. Suppose that  $v\in \oSob{1}{p_0}(\Omega)$ is a weak solution of \eqref{eq:non-divergence-type} {with $g=\Div \boldG$}.  If {$\frac{2n}{n+1}<r<2$}, then there exists a constant $C$ depending only on $\delta$, $K$, $n$ and $r$ such that 
\begin{align*}
  \int_{B_{R}(x_0)} |\nabla v|^2 \myd{x}&\leq C R^{n-2n/r} (1+\norm{\boldb}{\Lor{n}{\infty}(\Omega)}) \left(\int_{B_{2R}(x_0)}|\nabla v|^r \myd{x} \right)^{2/r}\\
  &\relphantom{=}+CR^{n/2-n/r} \left(\int_{B_{2R}(x_0)} |\nabla v|^r \myd{x} \right)^{1/r} \left(\int_{B_{2R}(x_0)} |\nabla v|^2 \myd{x} \right)^{1/2}\\
  &\relphantom{=}+\frac{1}{2}\int_{B_{2R}(x_0)} |\boldG|^2 \myd{x}+\left(\int_{B_{2R}(x_0)} |\boldG|^2\myd{x} \right)^{1/2}\left(\int_{B_{2R}(x_0)}|\nabla v|^2 \myd{x} \right)^{1/2}
\end{align*}
for all $B_{2R}(x_0)\subset \Omega$. 
\end{prop}
\begin{proof}
By Proposition \ref{prop:basic-estimates}, a standard density argument enables us to show that  \eqref{eq:weak-sol-nondiv} holds for all {$\phi \in \oSob{1}{p_0}(\Omega)$}.  Let  $B_{2R}(x_0)\subset \Omega$ and choose a cut-off function $\zeta \in C_c^\infty(B_{2R}(x_0))$ so that {$0\leq \zeta \leq 1$,} $\zeta =1$ on $B_R(x_0)${,} and $\norm{\nabla \zeta}{\Leb{\infty}(\mathbb{R}^n)}\leq C/R$ for some constant $C=C(n)>0$.  Denote $\overline{v}=v-(v)_{B_{2R}(x_0)}$ and take $\phi =\zeta^2 \overline{v}$ in \eqref{eq:weak-sol-nondiv}.  Since $\Div \boldb \geq 0$ in $\Omega$, it follows from Proposition \ref{prop:basic-estimates} that 
\[
-\int_{\Omega} (\boldb \cdot \nabla \overline{v})\zeta^2 \overline{v}\myd{x}-\int_{\Omega} (\boldb \cdot \nabla \zeta)\zeta \overline{v}^2 \myd{x}=-\int_{\Omega} [\boldb \cdot \nabla (\zeta \overline{v})]\zeta \overline{v}\myd{x}\geq 0.
\]
Then we get 
\begin{align*}
&\int_\Omega \zeta^2 (A^T\nabla {v})\cdot \nabla v \myd{x}+\int_\Omega \zeta \overline{v}^2 (\boldb \cdot \nabla \zeta)\myd{x}+2\int_\Omega \zeta \overline{v}(A^T\nabla v) \cdot \nabla \zeta \myd{x}\\
&\leq -\int_\Omega  \zeta^2 \boldG \cdot \nabla v \myd{x} -2\int_\Omega \zeta\overline{v}\boldG\cdot \nabla \zeta \myd{x} 
\end{align*} 

By H\"older's inequality in Lorentz spaces, we have 
\begin{align*}
\int_\Omega \zeta \overline{v}^2 (\boldb \cdot \nabla \zeta)\myd{x}&\leq \frac{C}{R} \norm{\boldb}{\Lor{n}{\infty}(\Omega)} \norm{|\overline{v}|^2}{\Lor{n'}{1}(B_{2R}(x_0))} \\
&= \frac{C}{R} \norm{\boldb}{\Lor{n}{\infty}(\Omega)} \norm{\overline{v}}{\Lor{2n'}{2}(B_{2R}(x_0))}^2 
\end{align*}
Hence it follows from H\"older's inequality in Lorentz spaces and Poincar\'e-Sobolev's inequality that there exists a constant {$C=C(n,r)>0$} such that 
\[  \norm{\overline{v}}{\Lor{2n'}{2}(B_{2R}(x_0))} \leq C R^{1+n/2n'-n/r}\norm{\nabla v}{\Leb{r}(B_{2R}(x_0))}  \]
Hence we get 
\begin{align*}     
\int_\Omega \zeta \overline{v}^2 (\boldb \cdot \nabla \zeta)\myd{x} &\leq   C R^{n-2n/r} \norm{\boldb}{\Lor{n}{\infty}(\Omega)} \norm{\nabla v}{\Leb{r}(B_{2R}(x_0))}^2
\end{align*}
for some constant $C=C(n,r)>0$. Similarly, it follows from the boundedness of $A$,  H\"older's inequality, and Poincar\'e-Sobolev's inequality that 
\begin{align*}
\int_{\Omega} \zeta \overline{v} (A^T\nabla v) \cdot \nabla \zeta \myd{x} &\leq \frac{C}{R} \int_{B_{2R}(x_0)} |\overline{v}| |\nabla v |\myd{x} \\
&\leq \frac{C}{R} \norm{\overline{v}}{\Leb{2}(B_{2R}(x_0))} \norm{\nabla v}{\Leb{2}(B_{2R}(x_0))}\\
&=CR^{n/2-n/r} \norm{\nabla v}{\Leb{r}(B_{2R}(x_0))} \norm{\nabla v}{\Leb{2}(B_{2R}(x_0))} 
\end{align*}
for some constant $C=C(n,r,K)>0$. Finally, by Cauchy-Schwarz inequality, we have 
\begin{align*} 
  -\int_{\Omega} \zeta^2 \boldG \cdot \nabla v \myd{x} &\leq \norm{\boldG}{\Leb{2}(B_{2R}(x_0))} \norm{\nabla v}{\Leb{2}(B_{2R}(x_0))}.
\end{align*}
Also, it follows from Young's inequality and Poincar\'e-Sobolev's inequality that 
\begin{align*}
\int_{\Omega} \zeta \overline{v} \boldG \cdot \nabla \zeta \myd{x} &\leq \frac{1}{2}\int_{B_{2R}(x_0)} |\boldG|^2 \myd{x} +\frac{C}{R^2} \int_{B_{2R}(x_0)} |\overline{v}|^2 \myd{x} \\
&\leq \frac{1}{2}\int_{B_{2R}(x_0)} |\boldG|^2 \myd{x} +\frac{C}{R^2} \norm{\nabla v}{\Leb{r}(B_{2R}(x_0))}^2 R^{n-2n/r+2} 
\end{align*}
for some constant $C=C(n,r)>0$. Collecting all of these estimates and using uniform ellipticity of $A$, we get the desired estimate. This completes the proof of Proposition \ref{prop:interior-reverse-Holder}.
\end{proof}

Next we derive a boundary reverse H\"older estimate for the gradient of weak solution. 
\begin{prop}\label{prop:boundary-reverse-Holder}
Let $\Omega$ be a bounded Lipschitz domain in $\mathbb{R}^n$  {and $\boldG\in \Leb{2}(\Omega)^n$}. Assume that $A$ satisfies \eqref{eq:uniformly-elliptic}, and $\boldb \in \Lor{n}{\infty}(\Omega)^n$, and $\Div \boldb \geq 0$ in $\Omega$.   Then {there exists $R_0>0$ depending on the Lipschitz character of $\Omega$ such that} if {$\frac{2n}{n+1}<r<2$} and $v\in \oSob{1}{p_0}(\Omega)$ is a weak solution of \eqref{eq:non-divergence-type} with {$g=\Div\boldG$}, then there exists a constant $C$ depending only on $\delta$, $K$, $n$,  $r$, and the Lipschitz character of $\Omega$ such that 
\begin{align*}
  \int_{\Omega_{R}(x_0)} |\nabla v|^2 \myd{x}&\leq C R^{n-2n/r} (1+\norm{\boldb}{\Lor{n}{\infty}(\Omega)}) \left(\int_{\Omega_{2R}(x_0)}|\nabla v|^r \myd{x} \right)^{2/r}\\
  &\relphantom{=}+CR^{n/2-n/r} \left(\int_{\Omega_{2R}(x_0)} |\nabla v|^r \myd{x} \right)^{1/r} \left(\int_{\Omega_{2R}(x_0)} |\nabla v|^2 \myd{x} \right)^{1/2}\\
  &\relphantom{=}+\frac{1}{2}\int_{\Omega_{2R}(x_0)} |\boldG|^2 \myd{x}+\left(\int_{\Omega_{2R}(x_0)} |\boldG|^2\myd{x} \right)^{1/2}\left(\int_{\Omega_{2R}(x_0)}|\nabla v|^2 \myd{x} \right)^{1/2}
\end{align*}
where $\Omega_R(x_0)=\Omega \cap B_R(x_0)$, $x_0 \in \partial{\Omega}$, and $0<R\leq R_0/2$.  
\end{prop}
\begin{proof}
Since $\Omega$ is a bounded Lipschitz domain in $\mathbb{R}^n$, it follows from \eqref{eq:Lipschitz-domain-exterior-ball} that there are constants $R_0>0$ and $C>0$ depending on the Lipschitz character of $\Omega$ such that 
\begin{equation*}
  |B_{2R}(x_0)\setminus \Omega| \geq CR^n\quad \text{for all } x_0 \in \partial\Omega,\quad 0< R\leq R_0/2.
\end{equation*} 
Extend $v$ and $\boldG$  by zero outside $\Omega$ and we denote them by $\tilde{v}$ and $\tilde{\boldG}$, respectively.  Since 
\[    \{ x \in B_{2R}(x_0) :  \tilde{v}(x)=0 \} \supset B_{2R}(x_0)\setminus  \Omega,\]
we have $|\{ x\in B_{2R}(x_0) : \tilde{v}(x)=0 \}|\geq C |B_{2R}(x_0)|$ for any $0<R\leq  R_0/2$.  Hence by H\"older's inequality in Lorentz spaces and Poincar\'e-Sobolev's inequality, we have 
\begin{align*}    
\norm{\tilde{v}}{\Lor{2n'}{2}(B_{2R}(x_0))} &\leq C R^{1/2+n/2-n/r} \norm{\tilde{v}}{\Leb{r^*}(B_{2R}(x_0))}\\
& \leq CR^{1+n/2n'-n/r}\norm{\nabla \tilde{v}}{\Leb{r}(B_{2R}(x_0))}.
\end{align*}
Choose a cut-off function $\zeta \in C_c^\infty(B_{2R}(x_0))$ so that $0\leq \zeta \leq 1$, $\zeta =1 $ on $B_{R}(x_0)$, and $\norm{\nabla \zeta}{\Leb{\infty}(\mathbb{R}^n)}\leq C(n)/R$. Then $\zeta^2 \tilde{v}\in \oSob{1}{2}(\Omega_{2R}(x_0))$. Hence following the exactly same argument as in Proposition \ref{prop:interior-reverse-Holder}, we can obtain the desired estimate. This completes the proof of Propsition \ref{prop:boundary-reverse-Holder}. 
\end{proof}

\section{Proofs of Theorem \ref{thm:main-theorem-1} and \ref{thm:main-theorem-2}}\label{sec:proofs}

This section is devoted to proving Theorems  \ref{thm:main-theorem-1} and \ref{thm:main-theorem-2} which are main theorems of this paper. We use the following lemma which can be found in \cite[Proposition 3.7]{DK18}. See  \cite[Chapter V]{G83} for the proof. 
\begin{lem}\label{lem:Gehring}
Let $1<q_0<q_1$, $\Phi\geq 0$ in a $n$-dimensional cube $Q$, and $\Psi \in \Leb{q_1}(Q)$. Suppose that 
\[  \fint_{B_r(x_0)} \Phi^{q_0}\myd{x}\leq C_0 \left(\fint_{B_{8r}(x_0)} \Phi \myd{x} \right)^{q_0} + C_0 \fint_{B_{8r}(x_0)} \Psi^{q_0}\myd{x}+\theta \fint_{B_{8r}(x_0)} \Phi^{q_0}\myd{x} \]
for every $x_0 \in Q$ and $0<r\leq R_2$ such that $B_{8r}(x_0)\subset Q$, where $R_2$ and $\theta$ are constants with $R_2>0$ and $\theta \in [0,1)$. Then $\Phi \in \Leb{p}_{\loc}(Q)$ for $p\in [q_0,q_0+\varepsilon)$ and 
\[  \left(\fint_{B_r(x_0)} \Phi^p\myd{x} \right)^{1/p}\leq C \left(\fint_{B_{8r}(x_0)} \Phi^{q_0}\myd{x} \right)^{1/q_0} + C \left(\fint_{B_{8r}(x_0)} \Psi^p\myd{x} \right)^{1/p} \]
for all $B_{8r}(x_0)\subset Q$, $0<r<R_2$, where $C$ and $\varepsilon$ depends only on $n$, $q_0$, $q_1$, $\theta$, and $C_0$, and $\varepsilon$ satisfies $0<\varepsilon<q_1-q_0$. 
\end{lem} 

Using interior and boundary reverse H\"older estimate for the gradient of weak solution, we obtain the following a priori estimate. 
\begin{thm}\label{thm:higher-integrability}
Let $\Omega$ be a bounded Lipschitz domain in $\mathbb{R}^n$  {$q>2$, and  $\boldG\in \Leb{q}(\Omega)^n$}. Assume that $A$ satisfies \eqref{eq:uniformly-elliptic}, $\boldb \in \Lor{n}{\infty}(\Omega)^n$, and $\Div \boldb \geq 0$ in $\Omega$.   Suppose that $v\in \oSob{1}{2}(\Omega)$ if $n\geq 3$ and $v\in \oSob{1}{p_0}(\Omega)$ if $n=2$ for some $p_0>2$ satisfies
\begin{equation}\label{eq:weak-solution-4}   
  \int_{\Omega} (A^T\nabla v) \cdot \nabla \phi-(\boldb \cdot \nabla v)\phi  \myd{x}=-\int_\Omega  \boldG \cdot \nabla \phi \myd{x} 
\end{equation}
for all $\phi \in C_c^\infty(\Omega)$. Then  there exists $\varepsilon>0$ depending on $n$, $\delta$, $K$, $\norm{\boldb}{\Lor{n}{\infty}(\Omega)}$ and the Lipschitz character of $\Omega$ such that  $\nabla v \in \Leb{2+\varepsilon}(\Omega)^n$. Moreover, we have 
\[  \norm{\nabla v}{\Leb{2+\varepsilon}(\Omega)} \leq C(\norm{\nabla v}{\Leb{2}(\Omega)} + \norm{\boldG}{\Leb{q}(\Omega)})\]
for some constant $C$ depending on $n$, $q$, $K$, $\norm{\boldb}{\Lor{n}{\infty}(\Omega)}$, and the Lipschitz character of $\Omega$. 
\end{thm}
\begin{proof}
Extend $v$ and $\boldG$ to be zero outside  $\Omega$. We write $\tilde{v}$ and $\tilde{\boldG}$ to be the extension of $v$ and $\boldG$, respectively. {If we set $r=\frac{2n+2}{n+2}$, then it follows from}  Propositions \ref{prop:interior-reverse-Holder}, \ref{prop:boundary-reverse-Holder}, and Young's inequality {that} for any $0<\theta<1$, there exists a constant  $C$ depending on $n$, $\theta$, $K$, $\delta$, and the Lipschitz character of $\Omega$ such that  
\begin{equation}\label{eq:reverse-type}
\begin{aligned}
\int_{B_{R}(x_0)} |\nabla \tilde{v}|^2 \myd{x}&\leq \theta\int_{B_{8R}(x_0)} |\nabla \tilde{v}|^2 \myd{x} \\
&\relphantom{=}+  C R^n (1+\norm{\boldb}{\Lor{n}{\infty}(\Omega)})\left(R^{-n} \int_{B_{8R}(x_0)} |\nabla \tilde{v}|^r \myd{x} \right)^{2/r}\\
&\relphantom{=}+C\int_{B_{8R}(x_0)}|\tilde{\boldG}|^2 \myd{x}
\end{aligned}
\end{equation}
for any $x_0 \in \mathbb{R}^n$ and $0<R\leq R_0/6$. Indeed, if $B_{2R}(x_0)\subset \Omega$, then the desired inequality follows by Proposition \ref{prop:interior-reverse-Holder}. If $B_{2R}(x_0)\cap \partial\Omega\neq \varnothing$, then there exists $y_0 \in \partial\Omega$ such that 
\[   |x_0-y_0|=\dist(x_0,\partial\Omega)\leq 2R. \]
This implies that 
\[   B_{R}(x_0)\subset B_{3R}(y_0)\subset B_{6R}(y_0)\subset B_{8R}(x_0). \]
It follows from Proposition \ref{prop:boundary-reverse-Holder} and $3R<R_0/2$ that the estimate \eqref{eq:reverse-type}  holds with $B_{3R}(y_0)$ and $B_{6R}(y_0)$ instead of $B_{R}(x_0)$ and $B_{8R}(x_0)$, respectively.   This implies the desired estimate \eqref{eq:reverse-type}. The case $B_{2R}(x_0)\subset \mathbb{R}^n\setminus \Omega$ is clear. This proves the estimate \eqref{eq:reverse-type}.

Set 
\[ \Phi = |\nabla \tilde{v}|^{r}\quad\text{and}\quad  \Psi = |\tilde{\boldG}|^r.  \]
Dividing $R^n$ in \eqref{eq:reverse-type}, we have  
\begin{align*}
\fint_{B_{R}(x_0)} \Phi^{2/r} \myd{x} &\leq \theta \left(\fint_{B_{8R}(x_0)} \Phi^{2/r}\myd{x} \right) \\
&\relphantom{=}+C(1+\norm{\boldb}{\Lor{n}{\infty}(\Omega)})\left(\fint_{B_{8R}(x_0)} \Phi \myd{x} \right)^{2/r}\\
&\relphantom{=}+C \fint_{B_{8R}(x_0)} \Psi^{2/r}\myd{x}.
\end{align*}
Since $2/r>1$, it follows from Lemma \ref{lem:Gehring} that there exists $p>2/r$ depending on $n$, $K$, $\delta$, and the Lipschitz character of $\Omega$ such that 
\[ \left(\fint_{B_{R}(x_0)} \Phi^p \myd{x} \right)^{1/p}\leq C\left(\fint_{B_{8R}(x_0)} \Phi^{2/r} \myd{x} \right)^{r/2}+C\left(\fint_{B_{8R}(x_0)} \Psi^{p} \myd{x} \right)^{1/p} \]
for some constant $C$ depending on $n$, $\norm{\boldb}{\Lor{n}{\infty}(\Omega)}$, $K$, $\delta$, and the Lipschitz character of $\Omega$. Since $pr>2$, it follows that there exists $\varepsilon>0$ such that $pr=2+\varepsilon$. This implies that 
\[  \left(\fint_{B_R(x_0)} |\nabla \tilde{v}|^{2+\varepsilon} \myd{x} \right)^{1/(2+\varepsilon)} \leq C\left(\fint_{B_{8R}(x_0)} |\nabla \tilde{v}|^2 \myd{x} \right)^{1/2} + C \left(\fint_{B_{8R}(x_0)} |\tilde{\boldG}|^q \myd{x} \right)^{1/q}    \]
Therefore by a standard covering argument, we have 
\[ \norm{\nabla v}{\Leb{2+\varepsilon}(\Omega)}\leq C (\norm{\nabla v}{\Leb{2}(\Omega)}+\norm{\boldG}{\Leb{q}(\Omega)})  \]
for some constant $C=C(n,q,\norm{\boldb}{\Lor{n}{\infty}(\Omega)},\delta,K,\Omega)>0$. This completes the proof of Theorem \ref{thm:higher-integrability}.
\end{proof}

Now we are ready to prove Theorem \ref{thm:main-theorem-1} when $n\geq 3$. 

\begin{proof}[Proof of Theorem \ref{thm:main-theorem-1} when $n\geq 3$]
Let $g\in \Sob{-1}{q}(\Omega)$. Then by Proposition \ref{prop:Riesz-representation}, there exists $\boldG \in \Leb{q}(\Omega)^n$ such that $\Div \boldG = g$ in $\Omega$. Since $q>2$ and $\Omega$ is bounded, it follows from Proposition \ref{prop:W12-estimates} that there exists a unique weak solution $v\in \oSob{1}{2}(\Omega)$ of \eqref{eq:non-divergence-type} satisfying \eqref{eq:weak-solution-4}. By Theorem \ref{thm:higher-integrability}, there exists $\varepsilon>0$ depending on $n$, {$q$,} $\delta$, $K$, $\norm{\boldb}{\Lor{n}{\infty}(\Omega)}$, and the Lipschitz character of $\Omega$ such that  $\nabla v \in \Leb{2+\varepsilon}(\Omega)^n$. Moreover, we have 
\[  \norm{\nabla v}{\Leb{2+\varepsilon}(\Omega)} \leq C(\norm{\nabla v}{\Leb{2}(\Omega)} + \norm{\boldG}{\Leb{q}(\Omega)})\]
for some constant $C$ depending on {$n$, $q$, $\Omega$,} $\delta$, $K$, $\norm{\boldb}{\Lor{n}{\infty}(\Omega)}$. Therefore, it follows from Proposition \ref{prop:W12-estimates} and the Poincar\'e inequality that  
\[   \norm{v}{\Sob{1}{2+\varepsilon}(\Omega)} \leq C \norm{\boldG}{\Leb{q}(\Omega)} \leq C \norm{g}{\Sob{-1}{q}(\Omega)}\]
for some constant $C$ depending on $\Omega$, $n$, $q$, $\delta$, $K$, $\norm{\boldb}{\Lor{n}{\infty}(\Omega)}$. This proves the existence part of Theorem \ref{thm:main-theorem-1}. The proof of uniqueness follows by Proposition \ref{prop:W12-estimates}. This completes the proof of Theorem \ref{thm:main-theorem-1}.  
\end{proof}

In contrast to the case of $n\geq 3$, further scrutiny is necessary when we are working on two-dimensional case since it seems hard to show an existence of weak solution of \eqref{eq:non-divergence-type} by using Lax-Milgram theorem as before. The main reason is that   $\boldb \cdot \nabla v$ does not belongs to $\Sob{-1}{2}(\Omega)$ for general $\boldb \in \Lor{2}{\infty}(\Omega)^2$ and $v\in \oSob{1}{2}(\Omega)$. To overcome this difficulty, we use some approximation argument by regularizing $A$ and $\boldb$, and domain approximation of  $\Omega$. 

The following proposition will be used to construct a solution of an approximate problem to \eqref{eq:non-divergence-type}. 
\begin{prop}\label{prop:approximation}
Let $\Omega$ be a bounded Lipschitz domain in $\mathbb{R}^n$, $n\geq 2$ and let $1<p<\infty$. Suppose that  $\boldb \in \Lor{p}{\infty}(\Omega)^n$ satisfies $\Div \boldb \geq 0$ in $\Omega$. Then there exist a sequence of domains $\{\Omega_k\}$ and a sequence of vector fields $\{\boldb_k\}$ such that 
\begin{enumerate}
\item[\rm (i)] there is a covering of $\partial\Omega$ by coordinate cylinders $Z$, so that given a coordinate pair $(Z,\varphi)$, then $Z^*\cap \partial\Omega_j$ is given as the graph of a $C^\infty$-function $\varphi_j$ for all $j$ such that $\varphi_j\rightarrow \varphi$ uniformly, $\norm{\nabla \varphi_j}{\Leb{\infty}(\mathbb{R}^{n-1})}\leq \norm{\nabla \varphi}{\Leb{\infty}(\mathbb{R}^{n-1})}$, and $\nabla \varphi_j\rightarrow \nabla \varphi$ pointwise a.e.;
\item[\rm (ii)] $\boldb_k \in C^\infty(\Omega)^n$ and $\boldb_k\rightarrow \boldb$ a.e. on $\Omega$ as $k\rightarrow \infty$. Also, $\norm{\boldb_{k}}{\Lor{p}{\infty}(\Omega)}\leq C \norm{\boldb}{\Lor{p}{\infty}(\Omega)}$, where $C$ depends only on $n$ and $p$. Moreover, 
\[   \norm{\boldb_k-\boldb}{\Leb{q}(\Omega)} \rightarrow 0 \]
as $k\rightarrow \infty$ for any $1<q<p$;
\item[\rm (iii)] for each $k \in \mathbb{N}$, we have $\Div \boldb_k\geq 0$ in $\Omega_k$.
\end{enumerate}
\end{prop}

\begin{proof}
Choose a sequence of bounded smooth domains $\{\Omega_k\}$ given in Theorem \ref{thm:Verchota}.  Extend $\boldb$ by zero outside $\Omega$ and denote the extension by $\overline{\boldb}$.  Let $\rho$ be a non-negative radial function such that $\supp \rho \subset B_1$ and $\int_{\mathbb{R}^n} \rho \myd{x}=1$. For $\varepsilon>0$, we define $\rho_{\varepsilon} (x) = \varepsilon^{-n} \rho(x/\varepsilon)$ and  $\boldb_{\varepsilon} =(\rho_{\varepsilon} * \overline{\boldb})$. Then $\boldb_{\varepsilon}\in C^\infty(\Omega)^n$. Also, $\boldb_{\varepsilon}\rightarrow \boldb$ a.e. on $\Omega$ and $\boldb_{\varepsilon}\rightarrow \boldb$ in $\Leb{q}(\Omega)^n$ for $1<q<p$. By Young's convolution inequality in weak space (see Theorem 1.4.25 in \cite{G} e.g.), we have
\[  \norm{\boldb_{\varepsilon}}{\Lor{p}{\infty}(\mathbb{R}^n)} \leq c\norm{\rho_\varepsilon}{\Leb{1}(\mathbb{R}^n)} \norm{\boldb}{\Lor{p}{\infty}(\mathbb{R}^n)} \leq c \norm{\boldb}{\Lor{p}{\infty}(\Omega)}, \]
where $c=c(n,p)$. 
Choose $0<\varepsilon_k <\frac{1}{2}\min\left\{\dist(\overline{\Omega}_k,\Omega^c),\frac{1}{k}\right\}$. Since $\overline{\Omega}_k\subset \Omega$, it is easy to see that $\Div \boldb_{\varepsilon_k}\geq 0$ in $\Omega_k$. Hence the sequence $\{\boldb_{\varepsilon_k}\}$ satisfies the desired properties (ii) and (iii). This completes the proof of Proposition \ref{prop:approximation}.
\end{proof} 

Now we are ready to prove Theorem \ref{thm:main-theorem-1} when $n=2$. 
\begin{proof}[Proof of Theorem \ref{thm:main-theorem-1} when $n=2$] 
Choose a sequence of domains $\{\Omega_k\}$ and a sequence of vector fields $\{\boldb_k\}$ which are given in Proposition \ref{prop:approximation}. Let $A_{\varepsilon}= (A*\rho_{\varepsilon})$, where $\{\rho_{\varepsilon}\}$ is a standard mollifier  given in a proof of Proposition \ref{prop:approximation}.  Then by Young's convolution inequality, we have $\norm{A_\varepsilon}{\Leb{\infty}(\mathbb{R}^2)}\leq \norm{\rho}{\Leb{1}(\mathbb{R}^2)}\norm{A}{\Leb{\infty}(\mathbb{R}^2)}$.  Note also that for any $\xi \in \mathbb{R}^n$, we have 
\[ \xi^T A_{\varepsilon }(x)\xi = \int_{\mathbb{R}^2}   [\xi^T A(x-y) \xi] \rho_{\varepsilon}(y) \myd{y} \geq \delta |\xi|^2 \int_{\mathbb{R}^2} \rho_{\varepsilon}(y)\myd{y}=\delta|\xi|^2.    \]

By the Lax-Milgram theorem, there exists a unique  $v_k\in \oSob{1}{2}(\Omega_k)$ satisfying 
\begin{equation}\label{eq:approximation-problem}   
\int_{\Omega_k} (A^T_{1/k} \nabla v_k)\cdot \nabla \phi -(\boldb_k\cdot \nabla v_k)\phi \myd{x}= -\int_{\Omega_k} \boldG \cdot \nabla \phi \myd{x} 
\end{equation}
for all $\phi \in C_c^\infty(\Omega_k)$. Moreover, we have 
\[   \norm{\nabla v_k}{\Leb{2}(\Omega_k)} \leq \norm{\boldG}{\Leb{2}(\Omega)} \leq C \norm{\boldG}{\Leb{q}(\Omega)},\]
where $C$ depends only on  $q$, and $|\Omega|$.  

On the other hand, since coefficients are smooth, it follows from the classical $\Leb{q}$-theory (see \cite{GT} e.g.) that there exists a unique $u_k \in \oSob{1}{q}(\Omega_k)$ satisfying \eqref{eq:approximation-problem}. Set $w_k=u_k-v_k$. Note that $w_k \in \oSob{1}{2}(\Omega_k)$ and $w_k$ satisfies \eqref{eq:approximation-problem} with $\boldG=\mathbf{0}$. Hence $\nabla w_k=\mathbf{0}$. By the Poincar\'e inequality, this implies that $w_k$ is identically zero in $\Omega_k$ and hence $v_k\in \oSob{1}{q}(\Omega_k)$.  

Since $q>2$, $\norm{\boldb_k}{\Lor{2}{\infty}(\Omega)}\leq C \norm{\boldb}{\Lor{2}{\infty}(\Omega)}$ for all $k$, and $\Omega_k$ satisfies the property in Proposition \ref{prop:approximation} (i),    it follows from Theorem \ref{thm:higher-integrability} that there exists $p>2$ depending on $n$, $\delta$, $K$, $\norm{\boldb}{\Lor{n}{\infty}(\Omega)}$ and the Lipschitz character of $\Omega$ such that 
\[   \norm{\nabla v_k}{\Leb{p}(\Omega_k)}\leq C \norm{\boldG}{\Leb{q}(\Omega)} \]
for some constant $C$ depending on $n$, $\delta$, $q$, $K$, and $\Omega$, which is uniform in $k$. Extend $v_k$ to be zero outside $\Omega_k$ and still denote it by $v_k$. Then $v_k \in \oSob{1}{p}(\Omega)$ and $v_k$ satisfies 
\[   \norm{\nabla v_k}{\Leb{p}(\Omega)} \leq C \norm{\boldG}{\Leb{q}(\Omega)} \]
for some constant $C$ independent of $k$. By Poincar\'e  inequality, the sequence $\{v_k\}$ is uniformly bounded in $\oSob{1}{p}(\Omega)$. Hence by the weak compactness result in $\oSob{1}{p}(\Omega)$, there exists a subsequence (which we still denote it by $\{v_k\}$) and $v\in \oSob{1}{p}(\Omega)$ such that $v_k\rightarrow v$ weakly in $\Sob{1}{p}(\Omega)$. Since $p>2$, we have  $\boldb_k \rightarrow \boldb$ in $\Leb{p'}(\Omega)$. Hence by Proposition \ref{prop:basic-estimates}, we have 
\[   \lim_{k\rightarrow \infty} \int_\Omega (A^T_{1/k}\nabla v_k)\cdot \nabla \phi- (\boldb_k \cdot \nabla v_k)\phi \myd{x} =\int_\Omega (A^T\nabla v)\cdot \nabla \phi -(\boldb\cdot \nabla v)\phi \myd{x} \]
for any $\phi \in C_c^\infty(\Omega)$. This implies that $v$ is a weak solution of \eqref{eq:non-divergence-type} with $g=\Div \boldG$. Moreover, $v$ satisfies 
\[ \norm{\nabla v}{\Leb{p}(\Omega)} \leq C \norm{\boldG}{\Leb{q}(\Omega)}  \]
for some constant $C$ depending on  $q$,  $\delta$, $K$, and $\norm{\boldb}{\Lor{2}{\infty}(\Omega)}$, and $\Omega$. To show the uniqueness part, suppose that $v\in \oSob{1}{p}(\Omega)$ is a weak solution satisfying 
\begin{equation}\label{eq:zero-solution}
  \int_\Omega (A^T\nabla v)\cdot \nabla \phi - (\boldb \cdot \nabla v)\phi \myd{x} =  0 
\end{equation} 
for all $\phi \in C_c^\infty(\Omega)$. By Proposition \ref{prop:basic-estimates}, \eqref{eq:zero-solution} holds for all $\phi \in \oSob{1}{2}(\Omega)$.  Since $\Div \boldb \geq 0$ in $\Omega$, it follows from Proposition \ref{prop:basic-estimates} that 
\[   -\int_{\Omega} (\boldb\cdot \nabla v)v\myd{x}\geq 0\quad \text{for all } v\in \oSob{1}{p}(\Omega).\]
Since $p>2$, letting $\phi=v$ in \eqref{eq:zero-solution}, we get 
\[  \delta \int_\Omega |\nabla v|^2 \myd{x}=0. \] 
Therefore, it follows from the Poincar\'e inequality that $v$ is identically zero in $\Omega$. This completes the proof of Theorem \ref{thm:main-theorem-1} when $n=2$. 
\end{proof}

Now we are ready to prove Theorem \ref{thm:main-theorem-2}.

\begin{proof}[Proof of Theorem \ref{thm:main-theorem-2}]
(i) Let $r=(2+\varepsilon)'$, where $\varepsilon>0$ appears in Theorem \ref{thm:main-theorem-1} corresponds to {$q=3$}. Suppose that $u\in \oSob{1}{r}(\Omega)$ satisfies \eqref{eq:weak-solution-uniqueness}. To show that $u=0$, let $\boldG \in C^\infty_c(\Omega)^n$ be given. Then by Theorem \ref{thm:main-theorem-1}, there exists a unique $\psi\in \oSob{1}{r'}(\Omega)$ such that 
\[   \int_\Omega (A^T \nabla \psi)\cdot \nabla \phi -(\boldb \cdot \nabla \psi)\phi \myd{x}=-\int_\Omega \boldG \cdot \nabla \phi \myd{x} \]
for all $\phi \in C^\infty_c(\Omega)$.  Moreover, we have
\[ \norm{\nabla v}{\Leb{2+\varepsilon}(\Omega)}\leq C\norm{\boldG}{\Leb{3}(\Omega)} \]
for some constant $C=C(n,\norm{\boldb}{\Lor{n}{\infty}(\Omega)},\delta,K,\Omega)>0$. Choose a sequence $\psi_k\in C^1(\overline{\Omega})$ with $\psi=0$ on $\partial\Omega$ satisfying $\psi_k\rightarrow \psi$ in $\Sob{1}{r'}(\Omega)$.  Then by Proposition \ref{prop:basic-estimates}, we can take $\phi=u$, and so we get 
\begin{align*}
-\int_\Omega \boldG \cdot \nabla u \myd{x}&=\int_\Omega (A^T\nabla \psi)\cdot \nabla u -(\boldb\cdot \nabla \psi)u\myd{x} \\
&=\lim_{k\rightarrow \infty} \int_\Omega (A^T\nabla \psi_k)\cdot \nabla u -(\boldb\cdot \nabla \psi_k)u\myd{x}=0.
\end{align*}
Since $\boldG \in C^\infty_c(\Omega)^n$, this implies that $\nabla u=\mathbf{0}$  a.e. on $\Omega$.  Hence by the Poincar\'e inequality, we conclude that $u$ is identically zero in $\Omega$. This completes the proof of (i).

(ii) Let $1<p<2$ and let $f\in \Sob{-1}{2-}(\Omega)$ be given. By Theorem \ref{thm:main-theorem-1}, there exists $2<s'<p'$ such that for every $g\in \Sob{-1}{p'}(\Omega)$, there exists a unique $v=Lg\in \oSob{1}{s'}(\Omega)$ satisfying 
\[   \int_\Omega (A^T \nabla v)\cdot \nabla \psi \myd{x}-\int_\Omega (\boldb \cdot \nabla v)\psi \myd{x}= \action{g,\psi} \]
for all $\psi \in C_c^\infty(\Omega)$. Moreover, we have 
\[ \norm{Lg}{\Sob{1}{s'}(\Omega)}\leq C \norm{g}{\Sob{-1}{p'}(\Omega)}  \]
for some constant $C>0$ depending on $n$, $p$, $\norm{\boldb}{\Lor{n}{\infty}(\Omega)}$, $\delta$, $K$, $\Omega$. 

Since $p<s<2$, it follows that $f\in \Sob{-1}{s}(\Omega)$. Define $\ell$ on $(\oSob{1}{p}(\Omega))'$ by $\ell(g)=\action{f,Lg}$. Then $\ell$ is linear and 
\[ \left|\action{f,Lg}\right|\leq \norm{f}{\Sob{-1}{s}(\Omega)}\norm{Lg}{\Sob{1}{s'}(\Omega)}\leq C \norm{f}{\Sob{-1}{s}(\Omega)} \norm{g}{\Sob{-1}{p'}(\Omega)}   \]
for some constant $C>0$ depending on $n$, $p$, $\norm{\boldb}{\Lor{n}{\infty}(\Omega)}$, $\delta$, $K$, $\Omega$. This implies that $\ell \in (\oSob{1}{p}(\Omega))''$. Since $\oSob{1}{p}(\Omega)$ is reflexive, it follows that there exists $u\in \oSob{1}{p}(\Omega)$ such that 
\[   \ell(g) = \action{g,u}\]
for all $g\in \Sob{-1}{p}(\Omega)$. Now given $\phi \in C_c^\infty({\Omega})$, define a linear functional $g$ by
\begin{equation}\label{eq:density-argument} 
     \action{g,\psi}=\int_\Omega (A^T\nabla \phi)\cdot \nabla \psi -(\boldb\cdot \nabla \phi) \psi \myd{x}\quad \text{for } \psi \in C_c^\infty(\Omega).
\end{equation}
By Proposition \ref{prop:basic-estimates}, a standard density argument shows that \eqref{eq:density-argument} holds for all $\psi \in \oSob{1}{p}(\Omega)$. So $g\in \Sob{-1}{p'}(\Omega)$. Then by definitions of $L$ and $\ell$, we have  $Lg=\phi$ and so 
\begin{equation}\label{eq:my-weak-solution}
\begin{aligned}
&\relphantom{=}\int_{\Omega} (A\nabla u)\cdot \nabla \phi -(u\boldb)\cdot \nabla \phi \myd{x}\\
&=\int_\Omega (A^T \nabla \phi)\cdot \nabla u -(\boldb\cdot \nabla \phi)u \myd{x}\\
&=\action{g,u} =\ell(g)=\action{f,Lg}=\action{f,\phi} 
\end{aligned}
\end{equation}
for all $\phi \in C_c^\infty(\Omega)$. This proves that given $f\in \Sob{-1}{2-}(\Omega)$ and $1<p<2$, there exists a function $u\in \oSob{1}{p}(\Omega)$ satisfying \eqref{eq:my-weak-solution}. Since $p<2$ is arbitrary chosen, it follows that $u\in \oSob{1}{2-}(\Omega)$. {This proves the existence of weak solution $u\in \oSob{1}{2-}(\Omega)$ satisfying \eqref{eq:divergence-type}. The uniqueness part of (ii) follows by (i).}  This completes the proof of Theorem \ref{thm:main-theorem-2}. 
\end{proof}

\bibliographystyle{amsplain}
\bibliography{Ref}

\end{document}